\documentclass[preprint,12pt]{elsarticle}

\usepackage{amssymb,color}
\usepackage{lineno,hyperref}
\modulolinenumbers[5]

\usepackage{amsmath,amsthm,amsfonts,amssymb,mathtools,mathrsfs,url}
\usepackage{paralist}
\usepackage{graphics} 
\usepackage{epsfig}   
\usepackage{graphicx}  \usepackage{epstopdf}

\newtheorem{theorem}{Theorem}[section]
\newtheorem{corollary}[theorem]{Corollary}
\newtheorem{lemma}[theorem]{Lemma}

\newtheorem{proposition}[theorem]{Proposition}
\newtheorem{definition}[theorem]{Definition}

\newtheorem{remark}[theorem]{Remark}
\numberwithin{equation}{section}
\newcommand{\BE}{\begin{equation}}
\newcommand{\EEQ}{\end{equation}}
\newcommand{\rfb}[1]{\mbox{\rm
    (\ref{#1})}\ifx\undefined\stillediting\else:\fbox{$#1$}\fi}

\newfont{\roma}{cmr10 scaled 1200}
\renewcommand{\cline}{{\mathbb C}}

\newcommand{\nline}  {{\mathbb N}}
\newcommand{\rline}  {{\mathbb R}}

\newcommand{\tline}  {{\mathbb T}}

\newcommand{\zline}  {{\mathbb Z}}

\newcommand{\PPP} {{\mathbf P}}

\newcommand{\dd}  {{\rm d}\hbox{\hskip 0.5pt}}
\renewcommand{\leq} {\leqslant}
\renewcommand{\geq} {\geqslant}
\newcommand{\mm}    {{\hbox{\hskip 0.5pt}}}
\newcommand{\m}     {{\hbox{\hskip 1pt}}}

\newcommand{\nm}    {{\hbox{\hskip -3pt}}}
\newcommand{\bluff} {{\hbox{\raise 15pt \hbox{\mm}}}}
\newcommand{\sbluff}{{\hbox{\raise 10pt \hbox{\mm}}}}
\newcommand{\Om}    {{\Omega}}

\newcommand{\e}      {{\varepsilon}}

\renewcommand{\l}    {{\lambda}}
\renewcommand{\o}    {{\omega}}
\newcommand{\FORALL} {{\hbox{$\hskip 11mm \forall \;$}}}
\newcommand{\rarrow} {\mathop{\rightarrow}}                 

\newcommand{\half}   {{\frac{1}{2}}}

\newcommand{\Leloc}    {{L^2_{\rm loc}[0,\infty)}}
\newcommand{\LE}[1]    {{L^2([0,\infty);#1)}}
\newcommand{\LEloc}[1] {{L^2_{\rm loc}([0,\infty);#1)}}
\font\bosy=cmbsy10
\def\conc{\hbox{\bosy \char '175}}

\newcommand\ICONT[1]{\mathop{\displaystyle \mathop{\conc}_{#1}}}

\newcommand{\Dscr} {\mathcal{D}}
\newcommand{\Hscr} {\mathcal{H}}

\newcommand{\Lscr} {\mathcal{L}}

\newcommand{\Oscr} {\mathcal{O}}
\newcommand{\Uscr} {\mathcal{U}}

\newcommand{\bbm}[1]{\left[\begin{matrix} #1 \end{matrix}\right]}
\newcommand{\sbm}[1]{\left[\begin{smallmatrix} #1
            \end{smallmatrix}\right]}
\begin{document}

\begin{frontmatter}

\title{Stabilizability properties of a linearized\\
       water waves system}


\author[mymainaddress]{Pei Su}
\ead{pei.su@u-bordeaux.fr}

\author[mymainaddress]{Marius Tucsnak\corref{mycorrespondingauthor}}
\cortext[mycorrespondingauthor]{Corresponding author}
\ead{marius.tucsnak@u-bordeaux.fr}
\ead[url]{https://www.math.u-bordeaux.fr/~mtucsnak/}

\author[mysecondaryaddress]{George Weiss}
\ead{                                       gweiss@tauex.tau.ac.il}

\address[mymainaddress]{Institut de Math\'ematiques de Bordeaux,\\
   Universit\'e de Bordeaux, \\
   351, Cours de la Lib\'eration - F 33 405 TALENCE, France}
\address[mysecondaryaddress]{School of Electrical Engineering,
   Tel Aviv University, Ramat Aviv 69978, Israel}

\begin{abstract}
We consider the strong stabilization of small amplitude gravity water
waves in a two dimensional rectangular domain. The control acts on one
lateral boundary, by imposing the horizontal acceleration of the water
along that boundary, as a multiple of a scalar input function $u$,
times a given function $h$ of the height along the active boundary.
The state $z$ of the system consists of two functions: the water level
$\zeta$ along the top boundary, and its time derivative $\dot\zeta$.
We prove that for suitable functions $h$, there exists a bounded
feedback functional $F$ such that the feedback $u=Fz$ renders the
closed-loop system strongly stable. Moreover, for initial states in
the domain of the semigroup generator, the norm of the solution decays
like $(1+t)^{-\frac{1}{6}}$. Our approach uses a detailed analysis of
the partial Dirichlet to Neumann and Neumann to Neumann operators
associated to certain edges of the rectangular domain, as well as
recent abstract non-uniform stabilization results by Chill, Paunonen,
Seifert, Stahn and Tomilov (2019).
\end{abstract}

\end{frontmatter}

\section{Notation} \label{notation} 

Throughout this paper, the notation
$$\nline,\ \zline,\ \rline,\ \cline$$
stands  for the sets of natural numbers (starting with $1$), integers,
real numbers and complex numbers, respectively. We denote $\zline^*=
\zline\setminus\{0\}$.

If $n,\m k\in\nline$ and $\Oscr\subset\rline^n$ is an open set, then
we use the notation $\Hscr^k(\Oscr)$ for the {\em Sobolev space}
formed by the distributions $f\in\Dscr'(\Oscr)$ having the property
that $\partial^\alpha f\in L^2(\Oscr)$ for every multi-index $\alpha
\in\zline^n$ with $\alpha_j\geq 0$ and $|\alpha|\leq k$. For $f\in
\Hscr^k(\Oscr)$ we set \vspace{-3mm}
\begin{equation} \label{Jeremy_is_out}
   \|f\|_k^2 \m=\m \sum_{|\alpha|\leq k} \|\partial^\alpha f\|^2
   _{L^2} \m.
\end{equation}
Let $\Hscr^0(\Oscr):=L^2(\Oscr)$ and let $\Hscr^s(\Oscr)$,
with $s>0$, denote the fractional order Sobolev spaces obtained by
interpolation via fractional powers of a positive operator (see, for
instance, Lions and Magenes \cite{LiMa}).

The system considered in this work is described by the linearized
equations of water waves in the vertical (in the sense of gravity)
rectangular domain \vspace{-2mm}
\begin{equation} \label{Omega_Mare}
   \Om \m=\m (0,\pi)\times (-1,0).
\end{equation}
We set \vspace{-2mm}
\begin{equation} \label{top_definition}
   \Hscr^1_{top}(\Om) = \{f\in\Hscr^1(\Om)\ \ |\ \ f(x,0)=0,\
   x\in(0,\pi)\},
\end{equation}
where the values at the top boundary are defined in the sense of
the Dirichlet trace, as in \cite{LiMa},
\cite[Sect.~13.6]{Obs_Book}).

For two functions $u$ and $v$ defined on $[0,\infty)$ and for any
$\tau\geq 0$, their \emph{$\tau$-concatenation}, denoted by
$u\ICONT\tau v$, is the function \vspace{-2mm}
$$ u\ICONT\tau v = \begin{cases} u(t)&\quad \hbox{for} \ \
   t\in[0,\tau),\\ v(t-\tau)&\quad \hbox{for}\ \ t\geq\tau .
   \end{cases}$$

If $H$ is a Hilbert space, $\Dscr(A_0)$ is a subspace of $H$ and
$A_0:\Dscr(A_0)\to H$ is a linear operator, then $A_0$ is called {\em
strictly positive} if $A_0$ is self-adjoint and there exists $m_0>0$
such that
$$ \langle A_0 z,z\rangle_H \m\geq\m m_0 \|z\|_H^2 \FORALL
   z\in\Dscr(A_0).$$

\section{The water wave model and its well-posedness}
\label{new_intro}

In this work we study the stabilizability of a system describing
small-amplitude water waves in a rectangular domain, in the presence
of a wave maker. For more details on water waves models we refer to
Whitham's book \cite[Chapter 13]{whitham2011linear} and to Lannes
\cite[Chapter 1]{lannes2013water}. Here we consider the stabilization
of linear water waves by an input (the acceleration of the wave maker)
acting at one of the lateral edges. We assume that the domain $\Om$ is
delimited at its top by a free water surface $\Gamma_s$ and that the
bottom $\Gamma_f$ is flat. The other two components of the boundary of
the fluid domain, denoted by $\Gamma_1$ and $\Gamma_2$, are supposed
to be vertical, see Figure \ref{fig1}. Moreover, we assume that the
fluid fills the rectangular domain $\Om$ defined in Section
\ref{notation}, that it is homogeneous, incompressible, inviscid and
that it undergoes irrotational flows. There is a wave maker that acts
at the left boundary of $\Om$, by injecting (or extracting) fluid in
the horizontal direction, at an acceleration determined by the control
signal $u$. \vspace{-4mm}

\begin{figure}[htbp]
   \centering\includegraphics[width=8cm]{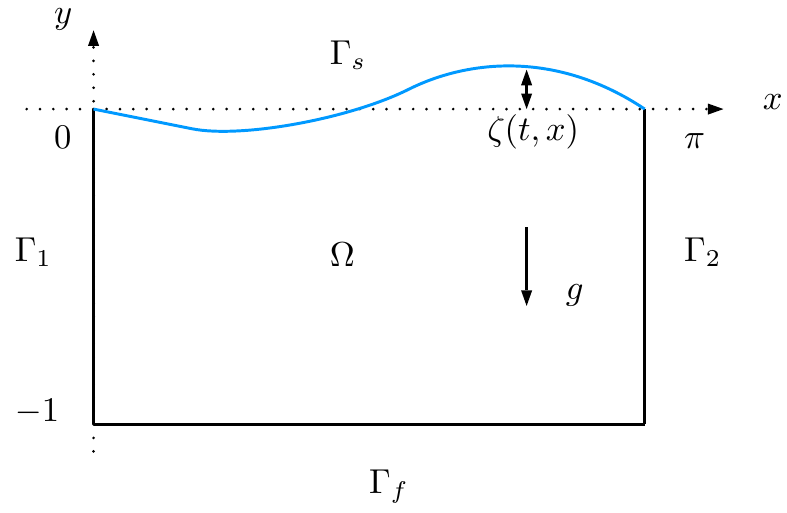}
   \caption{A rectangular domain $\Om$ filled with water} \label{fig1}
\end{figure} \vspace{-2mm}

The equations of the system, for all $t\geq 0$, are:
\begin{equation} \label{gravity-1}
\left\{ \begin{aligned}
   &\Delta\phi(t,x,y) \m=\m 0\quad&\quad (\ (x,y)\in\Om),\\
   &{\phi}(t,x,0)+\zeta(t,x) \m=\m 0\quad&\quad (\ x\in(0,\pi)),\\
   &\frac{\partial\phi}{\partial y}(t,x,0) \m=\m \ddot{\zeta}(t,x)
   \quad&\quad (\ x \in (0,\pi)),\\
   &\frac{\partial\phi}{\partial x}(t,0,y) \m=\m -h(y) u(t)\quad
   &\quad (\ y \in (-1,0)),\\ &\frac{\partial\phi}{\partial y}
   (t,x,-1)=0=\m \frac{\partial\phi}{\partial x}(t,\pi,y)\quad&\quad
   (\ (x,y)\in\Om). \end{aligned} \right.
\end{equation}
In the above equations $\phi$ stands for the derivative with respect
to time of the velocity potential of the fluid and $\zeta$ for the
elevation of the free surface. The function $h$ is given and it
represents the profile of the acceleration field imposed by the
wave maker. Usually we assume that $\int_{-1}^0 h(y)\dd y=0$, to ensure
the conservation of the volume of water. As far as we know, the
controllability and stabilizability properties of systems derived from
\rfb{gravity-1} have been first studied in Russell and Reid
\cite{reid1985boundary} and further in Mottelet
\cite{mottelet2000controllability}.

Before stating our well-posedness result for \rfb{gravity-1}, we need
some background and more notation. First we recall the concept of a
well-posed linear control system, following Weiss \cite{Weiss1}
(where these systems have been called {\em abstract linear control
systems}), see also Tucsnak and Weiss \cite{Obs_Book}.

\begin{definition} \label{WPS}
Let $U$ and $X$ be Hilbert spaces. A {\em well-posed linear control
system} with the state space $X$ and the input space $U$ is a couple
$(\tline,\Phi)$ of families of operators such that
\begin{enumerate}
\item $\tline=(\tline_t)_{t\geq 0}$ is a strongly continuous operator
semigroup (also called a $C_0$-semigroup) on $X$.

\item $\Phi=(\Phi_t)_{t\geq 0}$ is a family of bounded linear
operators from $L^2([0,\infty);U)$ to $X$ (called input maps) such
that for every $u,v\in L^2([0,\infty);U)$,
\begin{equation} \label{funceqPhi0}
   \Phi_{\tau+t} (u\ICONT{\tau}v) = \tline_t\Phi_\tau u + \Phi_t v
   \ \FORALL t,\tau\geq 0,
\end{equation}
where we used the concatenation of functions, see Section
\ref{notation}.
\end{enumerate}
\end{definition}

For any $\tau\geq 0$, let $\PPP_\tau u$ denote the truncation of $u:
[0,\infty)\rarrow U$ to $[0,\tau]$, setting $(\PPP_\tau u)(t)=0$ for 
$t>\tau$. It follows from \rfb{funceqPhi0} that $\Phi_\tau\PPP_\tau=
\Phi_\tau$ (causality), and hence $\Phi_\tau$ has a natural extension
to $\LEloc U$.

Still using the notation from the above definition, if $z_0\in X$ and
$u\in\LEloc U$, then we call the function $z(t)=\tline_t z_0+\Phi_t u$
the {\em state trajectory} of the system corresponding to the initial
state $z_0$ and the input $u$. Let $A:\Dscr(A)\rarrow X$ denote the
generator of $\tline$. For every well-posed linear control system
there exists a (usually unbounded) operator $B$ defined on $U$ and
with values in an extrapolation space that contains $X$, with the
following property: For any $z_0\in X$ and $u\in\LEloc U$, the
corresponding state trajectory is the unique solution (in the
extrapolation space) of the abstract differential equation
\begin{equation} \label{Ax+Bu}
   \dot z(t) \m=\m Az(t)+Bu(t) \m,
\end{equation}
with initial condition $z(0)=z_0$. For details on this see
\cite[Chapter 4]{Obs_Book}. The above operator $B$ is called the {\em
control operator} of the system. This operator is called {\em bounded}
if $B\in\Lscr(U,X)$ (this is the case of interest in this paper).

We would like to formulate the system of equations \rfb{gravity-1} as
a well-posed linear control system. This is not obvious, because the
equations \rfb{gravity-1} do not even resemble \rfb{Ax+Bu}. We have to
define what we mean by the state of our system at some time $t\geq 0$:
this should be
\begin{equation} \label{Salami}
   z(t) \m=\m \bbm{\zeta(t,\cdot)\\ \dot\zeta(t,\cdot)} .
\end{equation}

To define the state space $X$, and also for other arguments, we
introduce a scale of Hilbert spaces as follows. We set
\begin{equation} \label{H}
   H \m=\m \left\{\eta\in L^2[0,\pi]\ \ \left|\ \ \int_0^\pi \eta(x)\,
   \dd x=0\right\}\right.,
\end{equation}
which is a Hilbert space when endowed with the inner product inherited
from $L^2[0,\pi]$. It is known that the family $(\varphi_k)_{k\in
\nline}$ defined by
\begin{equation} \label{mare_definitie}
   \varphi_k(x) \m=\m \sqrt{\frac{2}{\pi}}\cos(kx)\ \FORALL
   x\in [0,\pi],
\end{equation}
forms an orthonormal basis in $H$. For any $\eta\in H$, we denote 
$\eta_k=\langle\eta,\varphi_k\rangle$. The scale of Hilbert spaces 
$(H_\alpha)_{\alpha\geq 0}$ are defined by $H_0=H$ and
\begin{equation} \label{HALPHA}
   H_\alpha \m=\m \left\{\eta\in H\ \ \left|\ \ \sum_{k\geq 1}
   k^{2\alpha} |\eta_k|^2 < \infty\right\}\right. \qquad 
   (\alpha\geq 0),
\end{equation}
with the inner products $\left(\langle\cdot,\cdot\rangle_\alpha
\right)_{\alpha\geq 0}$ defined by \m $\langle\eta,\psi\rangle_\alpha=
\sum_{k\geq 1}k^{2\alpha}\eta_k\m\overline{\psi_k}$, for all $\eta,\m
\psi\in H_\alpha$. It is not difficult to check that
$$ H_1 \m=\m \left\{\eta\in\Hscr^1(0,\pi)\ \ \left|\ \ \int_0^\pi
   \eta(x)\, \dd x = 0\right\} \right. .$$
By interpolation theory (see, for instance, Lions and Magenes
\cite{LiMa}, Bensoussan et al \cite[Part II]{BDDM} and Chandler-Wilde
et al \cite{chandler2015interpolation}) it follows that
\begin{equation} \label{interpolation}
   H_s = \left\{\eta\in\Hscr^s(0,\pi)\ \ \left|\ \ \int_0^\pi
   \eta(x)\, \dd x=0 \right\}\right. \ \FORALL s\in (0,1).
\end{equation}

We define what we mean by a solution of the water wave equations
\rfb{gravity-1}.

\begin{definition} \label{def_solutii}
Given $u\in\Leloc$ and $h\in L^2[-1,0]$, with $\int_{-1}^0 h(y)\,
\dd y=0$, a couple $(\phi,\zeta)$ is called a {\em solution} of
\rfb{gravity-1} if \vspace{-1mm}
$$ \phi\in\LEloc{\Hscr^1(\Om)},\quad \zeta\in C([0,\infty);H_\half)
   \cap C^1([0,\infty);H),\vspace{-2mm}$$
\vspace{-3mm}
\begin{equation} \label{Jeremy_C}
   \phi(t,\cdot,0)+\zeta(t,\cdot) \m=\m 0 \m, {\rm\ equality\ in\ }
   \LEloc{L^2[0,\pi]},
\end{equation}
and for every $\Psi\in\Hscr^1(\Om)$ and every $t>0$ we have
\vspace{-1mm}
\begin{equation} \label{weaksol}
   \int_0^\pi \dot\zeta(t,x)\overline{\Psi(x,0)}\m\dd x -\int_0^\pi
   \dot\zeta(0,x) \overline{\Psi(x,0)}\m\dd x \m=\m \qquad\qquad\m  
\end{equation} \vspace{-2mm}
$$ \int_0^t\int_\Om \nabla\phi(\sigma,x,y) \cdot \overline{\nabla
   \Psi(x,y)}\m\dd x\,\dd y\, \dd\sigma - \int_0^t u(\sigma)
   \int_{-1}^0 h(y)\overline{\Psi(0,y)} \m\dd y\, \dd \sigma.$$
\end{definition}

\begin{remark} \label{baccalaureat} {\rm
We explain the connection between the water waves equations
\rfb{gravity-1} and their variational formulation \rfb{Jeremy_C}-
\rfb{weaksol}. In one direction, assume that $(\phi,\zeta)$ is a 
classical solution of \rfb{gravity-1}, having the smoothness
\begin{equation} \label{Boris_J}
   \phi\in C([0,\infty);\Hscr^2(\Om)) \m,\qquad
   \zeta\in C(([0,\infty);H_\half) \cap C^2([0,\infty);H) \m.
\end{equation}
(If $u\not\equiv0$, this implies that $h\in\Hscr^\half(-1,0)$ and $u$
is continuous.) The equation \rfb{Jeremy_C} is simply copied from
\rfb{gravity-1}. We multiply the first equation in \rfb{gravity-1}
with $\overline\Psi$ and apply the first Green formula (integration by
parts), taking into account the last three lines of \rfb{gravity-1}.
After this we do simple integration with respect to $t$, and we obtain
\rfb{weaksol}.

In the opposite direction, let us assume that $(\phi,\zeta)$ is a
solution of \rfb{Jeremy_C}-\rfb{weaksol} with the additional
regularity \rfb{Boris_J}, and $u$ is continuous. Then \rfb{weaksol}
can be differentiated with respect to the time $t$, and after using
the first Green formula we obtain
$$ \int_0^\pi \ddot\zeta(t,x)\overline{\Psi(x,0)}\m\dd x \m=\m
   \int_{\partial\Om} \left( \frac{\partial}{\partial\nu} \phi\right)
   \overline\Psi \dd\sigma - \int_\Om \Delta\phi(t,x,y) \cdot
   \overline{\Psi(x,y)} \m\dd x\,\dd y$$
$$ \m\qquad\qquad - u(t)\int_{-1}^0 h(y)\overline{\Psi(0,y)} \m\dd y
   \FORALL t\geq 0 \m,$$
where $\frac{\partial}{\partial\nu}$ denotes the Neumann trace on the
entire boundary $\partial\Om$. Considering only functions $\Psi$ with
compact support in $\Om$, we see from the above that we must have
$\Delta\phi=0$. After this, we consider test functions $\Psi$ whose
trace is supported on one of the four segments of $\partial\Om$,
knowing that these traces are dense in the $L^2$ space of the
relevant segment, see \cite[Theorem 13.6.10]{Obs_Book}. From here
we can get that $\phi$ and $\zeta$ satisfy also the last three
equations in \rfb{gravity-1}. (It also follows that $h\in\Hscr^\half
(-1,0)$.) }
\end{remark}

The following result establishes the existence of a well-posed linear
control system corresponding to \rfb{gravity-1}. For the proof we
refer to Section \ref{sem_set}.

\begin{theorem} \label{give_context}
Let $h\in L^2[-1,0]$ be such that $\int_{-1}^0 h(y)\, \dd y=0$. Then
for every $u\in\Leloc$, $\zeta_0\in H_\half$ and $w_0\in H$, there
exists a unique solution of \rfb{gravity-1} with $\zeta(0)=\zeta_0$
and $\dot\zeta(0)=w_0$. Moreover, there exists a well-posed linear
control system $(\tline,\Phi)$ with state space $X=H_\half\times H$
and input space $U=\cline$ such that, setting $z_0=\sbm{\zeta_0\\
w_0}$ and using the state from \rfb{Salami}, we have \vspace{-1mm}
\begin{equation} \label{pancakes}
   z(\tau) \m=\m \tline_\tau z_0 + \Phi_\tau u \FORALL \tau\geq 0.
\end{equation}
Finally, the generator $A$ of \m $\tline$ is skew-adjoint, with domain
$\Dscr(A)=H_1\times H_\half$, and there exists $B\in\Lscr(\cline,X)$
such that for any $\tau\geq 0$, \vspace{-1mm}
\begin{equation} \label{Erdogan}
   \Phi_\tau u \m=\m \int_0^\tau \tline_{\tau-\sigma} Bu(\sigma)\,\dd
   \sigma \FORALL u\in\Leloc.
\end{equation}
\end{theorem}

We mention that, according to the above theorem and what we have
said around \rfb{Ax+Bu}, the state strajectories of our system
are solutions of \rfb{Ax+Bu}, in the sense of \cite[Sect.~4.1-4.2]
{Obs_Book}, and our control operator $B$ is bounded.

\section{Statement of the main result} \label{main} 

We recall some commonly used stabilizability concepts, for the
particular situation of bounded control and feedback operators.

\begin{definition} \label{def_stab}
Let $\Sigma=(\tline,\Phi)$ be a well-posed linear control system with
state space $X$ and input space $U$. Let $A$ be the generator of
$\tline$ and assume that there exists $B\in\Lscr(U,X)$ such that
\rfb{Erdogan} holds. For some feedback operator $F\in\Lscr(X,U)$ we 
denote by $\tline^{cl}$ the (closed loop) operator semigroup on $X$
generated by $A+BF$. Then the system $(\tline,\Phi)$ is:
\begin{enumerate} \item {\em Exponentially stabilizable} with bounded 
   feedback, if there exists $F\in\Lscr(X,U)$ such that the semigroup 
   $\tline^{cl}$ is exponentially stable;
   \item {\em Strongly stabilizable} with bounded feedback, if there 
   exists $F\in\Lscr(X,U)$ such that the semigroup $\tline^{cl}$ is 
   strongly stable;
   \item {\em Uniformly stabilizable for smooth data} (USSD), if there
   exists $F\in\Lscr(X,U)$ and $f:[0,\infty)\to [0,\infty)$, with 
   $\displaystyle\lim_{t\to\infty} f(t)=0$, such that
\begin{equation} \label{USSD}
   \|\tline^{cl}_t z_0\|_X\leq f(t)\|z_0\|_{\Dscr(A)} \FORALL z_0\in
   \Dscr(A),\ t\geq 0.
\end{equation}
\end{enumerate}
If $f$ in \rfb{USSD} can be chosen such that $\displaystyle\lim_{t
\to\infty}t^m f(t)=0$ for some $m\in\nline$, then the USSD property is
called {\em polynomial stabilizability}.
\end{definition}

In \rfb{USSD} and also later, $\|\cdot\|_{\Dscr(A)}$ denotes the
graph norm on $\Dscr(A)$.

\begin{remark} \label{digression} {\rm
Note that the property \rfb{USSD} does not imply that the semigroup
$\tline^{cl}$ is strongly stable. Indeed, consider $\tline^{cl}_t$ to
be ${\rm e}^{-0.7t}$ times the semigroup from \cite[Example 2.3]
{Weiss4} (based on Zabczyk \cite{Zabcz}), with $\l_n=2^n$, then it
satisfies \rfb{USSD} with $f(t)=M{\rm e}^{-0.2t}$ (for some $M>0$) but
$\tline^{cl}$ is exponentially growing: $\|\tline^{cl}_t\|={\rm
e}^{0.3t}$. However, if $\tline^{cl}$ is a bounded semigroup and
\rfb{USSD} holds, then it is easy to see that $\tline^{cl}$ is
strongly stable. }
\end{remark}

Here is our main result:

\begin{theorem} \label{th_mergebine}
Let $\Sigma=(\tline,\Phi)$ be the well-posed linear control system
introduced in Theorem {\rm\ref{give_context}}. Then
\begin{enumerate}
\item $\Sigma$ is not exponentially stabilizable with bounded 
    feedback;
\item $\Sigma$ is strongly stabilizable with bounded feedback if and
    only if $h$ is a {\em strategic profile,} in the sense that 
    \vspace{-2mm}
    \begin{equation} \label{B*phi_start}
       \int_{-1}^0 h(y) \cosh{[k(y+1)]}\, \dd y \neq 0 \FORALL
       k\in\nline;
    \end{equation}
    In this case, one strongly stabilizing feedback operator is
    $F=-B^*$.
\item If \vspace{-2mm}
    \begin{equation} \label{gasitabine}
       \inf_{k\in\nline} \m \frac{k}{\cosh k} \left| \int_{-1}^0
       h(y) \cosh{[k(y+1)]} \,\dd y\right| \m>\m 0,
    \end{equation}
    then  the system $\Sigma$ is USSD. More precisely, the feedback
    operator $F=-B^*$ leads to the closed-loop semigroup
    $\tline^{cl}$ (with generator $A-BB^*$) which is strongly
    stable and has the following property: there exists $M>0$ such
    that
    \begin{equation} \label{final_estimate}
       \|\tline_t^{cl} z_0\|_X \m\leq\m \frac{M}{(1+t)^{\frac16}}
       \|z_0\|_{\Dscr(A)} \FORALL z_0\in\Dscr(A),\ t\geq 0.
    \end{equation}
\end{enumerate}
\end{theorem}

\begin{remark} {\rm
It is not difficult to check (by integration by parts) that condition
\rfb{gasitabine} is satisfied, for instance, if there exists
$\e\in(0,1)$ such that
\begin{equation} \label{SC}
   \lVert h'\rVert_{L^\infty[-1,0]} \m<\m \frac{(1-\e)
   \tanh{1}}{1-\frac{2}{{\rm e}}}|h(0)|,
\end{equation}
where ${\rm e}=2.71828...$ is the basis of the natural
logarithm. Indeed, there are many functions satisfying \rfb{SC} and
$\int_{-1}^0 h(y)\dd y=0$, such as the linear function $h_1(y)=y+
\half$, the trigonometric function $h_2(y)=\cos{\left[\half\pi(y+
\frac{3}{2})\right]}$ and some slightly modified step functions.
Compared with other strategic conditions, for instance the constraint
condition at rational points in \cite{ammari2000stabilization}, the
condition \rfb{SC} is easier to satisfy in practice.}
\end{remark}

\begin{remark} {\rm
The first two conclusions in Theorem \ref{th_mergebine} appear
partially in \cite{mottelet2000controllability}, with some steps of
the proof not given. (For instance the operators $A_0$ and $B_0$ that
we introduce in Section \ref{sec5} are used without a detailed
construction and proof of their main properties.) As far as we know,
the property of the water waves system described in the third point of
Theorem \ref{th_mergebine} is new, and gives us more detailed 
information on the stability of the closed-loop system.}
\vspace{-8mm}
\end{remark}

\section{Some background on the partial Dirichlet and Neumann maps
         in a rectangular domain } \label{sec4} 

\vspace{-2mm}
In this section we consider two boundary value problems for the
Laplacian in the rectangular domain $\Om=(0,\pi)\times(-1,0)$ and we
define the corresponding solution operators. Note that, $\Om$ being a
rectangle, we are able to construct these solution operators, as well
as the Dirichlet to Neumann and Neumann to Neumann operators (in the
next section) in an elementary and explicit way, using the separation
of variables and analysis of Fourier or Dirichlet series.  Another
possible approach to these issues, pursued in
\cite{mottelet2000controllability}, is the use of the much more
sophisticated theory of elliptic problems in polygonal domains as
described, for instance, in Grisvard \cite{Grisvard}.

We begin by introducing a self-adjoint operator on $L^2(\Om)$ which
plays an important role in our arguments in this section.

\begin{proposition} \label{AZERO0}
With $\Om$ as in \rfb{Omega_Mare}, we consider the operator
$A_1:\Dscr(A_1)\to L^2(\Om)$ defined by \vspace{-3mm}
$$ \Dscr(A_1) = \left\{f \in \Hscr^2(\Om)\ \ \left|\ \ \begin{matrix}
   f(x,0)=0,\ \frac{\partial f}{\partial y}(x,-1)=0& x\in (0,\pi)\\
   \frac{\partial f}{\partial x}(0,y)=0,\ \frac{\partial f}
   {\partial x}(\pi,y)=0 & y\in (-1,0) \end{matrix}\right.\right\},$$
$$ A_1 f = -\Delta f \FORALL f\in\Dscr(A_1).$$
Then $A_1$ is a strictly positive operator on $L^2(\Om)$.
\end{proposition}

\begin{proof}
The operator $A_1$ is obviously symmetric. Moreover, the family
\vspace{-2mm}
$$ \Psi_{kl}(x,y) = \frac{2}{\sqrt{\pi}}\cos(kx)\sin\left[(2l-1)
   \frac{\pi}{2}y\right] \FORALL k,l\in\nline,\ (x,y)\in\Om,$$
is an orthonormal basis for $L^2(\Om)$ formed of eigenvectors of
$A_1$, corresponding to the eigenvalues \vspace{-2mm}
$$\m\ \l_{kl}=k^2+(2l-1)^2\frac{\pi^2}{4} \FORALL k,l\in\nline.$$
Let $g\in L^2(\Om)$, so that \m $g=\sum_{k,l\in\nline} c_{kl}
\Psi_{kl}$, with $c_{kl}\in l^2(\nline^2)$. This implies that $f$
defined by
$$ f=\sum_{k,l\in\nline} \frac{c_{kl}}{k^2+(2l-1)^2\frac{\pi^2}{4}}
   \Psi_{kl},$$
satisfies $f\in\Dscr(A_1)$ and $A_1f=g$. Thus the operator $A_1$ is
onto so that (see, for instance, \cite[Proposition 3.2.4]{Obs_Book})
$A_1$ is self-adjoint. Finally, it follows from the first Green
formula that
$$ \langle A_1 f,f\rangle_{L^2(\Om)} \m=\m \|\nabla f\|_{L^2(\Om)}^2
   \FORALL f\in\Dscr(A_1) \m.$$
This, together with a version of the Poincar\'e inequality (see
\cite[Theorem 13.6.9]{Obs_Book}), implies that $A_1$ is strictly
positive.
\end{proof}

\begin{proposition} \label{def-D}
For every $\eta\in L^2[0,\pi]$ there exists a unique function
$D\eta\in L^2(\Om)$ such that \vspace{-1mm}
\begin{equation} \label{taking}
   \int_\Om (D\eta)(x,y) \overline{g(x,y)}\m \dd x \m \dd y \m=
   -\int_0^\pi \eta(x) \overline{\frac{\partial(A_1^{-1}g)}{\partial
   y}(x,0)}\, \dd x \ \ \forall\m g\in L^2(\Om).
\end{equation}
Moreover, the operator $\eta\mapsto D\eta$ (called the {\em
partial Dirichlet map}) is bounded from $L^2[0,\pi]$ into $L^2(\Om)$.
\end{proposition}

\begin{proof}
We first note from Proposition \ref{AZERO0} that $A_1^{-1}\in
\Lscr(L^2(\Om),\Hscr^2(\Om))$. Thus, by a standard trace theorem the
map $g\mapsto \frac{\partial(A_1^{-1}g)}{\partial y}(\cdot,0)$ is
bounded from $L^2(\Om)$ to $L^2[0,\pi]$. Consequently, the right-hand
side of \rfb{taking} defines an antilinear functional of the argument
$g\in L^2(\Om)$, and the result follows by applying the Riesz
representation theorem. (See also \cite[Sect.~10.6]{Obs_Book}.)
\end{proof}

\begin{remark} \label{Ephraim_Mirvis} {\rm
For every $\eta\in H$, we have $D\eta\in C^\infty(\Om)$ and
$\Delta(D\eta)=0$. Indeed, this follows by an argument that is similar
to the one used in the proof of \cite[Proposition 10.6.2]{Obs_Book}:
We take $g=\Delta\varphi$ with $\varphi\in\Dscr(\Om)$ in \rfb{taking}
to see that $\Delta(D\eta)=0$ in the sense of distributions. It
follows from \cite[Remark 13.5.6]{Obs_Book} that $D\eta\in\Hscr^n
_{loc}(\Om)$ for every $n\in\nline$. Then we use the embedding
$\Hscr_{\rm loc}^n(\Om)\subset C^m(\Om)$ for $n>1+m$ ($m\in\nline$)
(see \cite[Remark 13.4.5]{Obs_Book}), so that indeed $D\eta\in 
C^\infty(\Om)$, and hence $\Delta(D\eta)=0$. Moreover, if $D\eta\in 
C^1(\overline\Om)$, then $D\eta$ is the unique function
in $C^2(\Om)\cap C(\overline\Om)$ that satisfies, in the classical
sense, the following boundary value problem: \vspace{-1mm}
\begin{equation} \label{LaplaceDeta}
   \left\{ \begin{aligned} &\Delta (D\eta)(x,y) \m=\m 0 \quad&\quad
   ((x,y)\in\Om),\\ &(D\eta)(x,0) \m=\m \eta(x),\quad \frac{\partial
   (D\eta)}{\partial y}(x,-1) \m=\m 0\quad&\quad \qquad (x\in(0,\pi)),
   \\ &\frac{\partial (D\eta)}{\partial x}(0,y) \m=\m 0, \quad \frac
   {\partial (D\eta)}{\partial x}(\pi,y) \m=\m 0\quad&\quad\qquad
   (y\in (-1,0)). \end{aligned} \right.
\end{equation}
To see this, we take in \rfb{taking} $g=\Delta f$, where $f\in\Dscr
(A_1)$, and use integration by parts, which yields that 
$$ \int_0^\pi \eta(x)\overline{\frac{\partial f}{\partial y}(x,0)}
   \dd x \m=\m \int_0^\pi \left[ (D\eta)\overline{\frac{\partial f}
   {\partial y}} \right](x,0) \dd x \qquad \qquad \qquad \qquad \m$$
$$ + \int_0^\pi \left[\frac{\partial D\eta}{\partial y}\overline{f}
   \right](x,-1) \dd x + \int_{-1}^0 \left[\frac{\partial D\eta}
   {\partial x}\overline{f}\right](0,y)\dd y - \int_{-1}^0 \left[
   \frac{\partial D\eta}{\partial x}\overline{f}\right](\pi,y)\dd y
   \m.$$
If we choose $f\in\Dscr(A_1)$ such that $f=0$ on the lateral
boundaries and the bottom of $\Om$, we obtain that $(D\eta)(x,0)=\eta
(x)$ for almost every $x\in[0,\pi]$. By choosing suitable other test
functions $f\in \Dscr(A_1)$ (we omit the details), we can obtain also
the remaining three equalities in \rfb{LaplaceDeta}.}
\end{remark}

\begin{remark} {\rm
The term ``partial Dirichlet map'' comes from the fact that $D$
acts on the upper boundary of $\Om$ rather than the entire boundary
$\partial\Om$.}
\end{remark}

\begin{lemma} \label{Deta}
For every $\eta\in H$, $D\eta$ is given by
\begin{equation} \label{Detaformula}
   (D\eta)(x,y) = \sum_{k\in\nline} \frac{\langle\eta,\varphi_k
   \rangle}{\cosh{k}}\varphi_k(x)\cosh{[k(y+1)]} \FORALL x,y\in\Om,
\end{equation}
where the functions $\varphi_k$ have been introduced in
\rfb{mare_definitie}. Moreover, for every $\eta\in H_3$ we have
$D\eta\in C^2(\overline\Om)$.
\end{lemma}

\begin{proof}
Using Remark \ref{Ephraim_Mirvis} it is easily checked that for every
$k\in\nline$ we have
\begin{equation} \label{Dformula}
   (D\varphi_k)(x,y) = \sqrt{\frac{2}{\pi}}\frac{\cos{(kx)}
   \cosh{[k(y+1)]}}{\cosh{(k)}} \FORALL x,y\in\Om.
\end{equation}
On the other hand, we can see that the right-hand side of
\rfb{Detaformula} converges in $L^2(\Om)$. This fact, together with
\rfb{Dformula} clearly implies \rfb{Detaformula}.

Moreover, for every $\alpha\in \{0,1,2\}$ we have
$$ \left|\frac{\partial^{\alpha,2-\alpha}}{\partial x^\alpha
   \partial y^{2-\alpha}}\left(\frac{\cos{(kx)}\cosh{(k(y+1))}}
   {\cosh{(k)}}\right)\right|\leq k^2 \FORALL k\in\nline,\
   x,y\in\Om.$$
Using the Cauchy-Schwarz inequality and the fact that
$\sum_{k\in\nline}1/k^2=\pi^2/6$,
\begin{multline*}
   \sum_{k\in\nline} \left|\frac{\partial^{\alpha,2-\alpha}}
   {\partial x^\alpha\partial y^{2-\alpha}}\left(\frac{\langle\eta,
   \varphi_k\rangle}{\cosh{k}}\varphi_k(x)\cosh{(k(y+1))}\right)
   \right| \\ \leq\m \sum_{k\in\nline} \frac{1}{k} \cdot k^3 \left|
   \langle\eta,\varphi_k\rangle\right| \m\leq\m \frac{\pi}{\sqrt 6}
   \m\|\eta\|_{H_3} \FORALL \eta\in H_3,\ x,y\in\Om.
\end{multline*}
Combining the last estimate with \rfb{Detaformula}, we obtain that
indeed $D\eta\in C^2(\overline\Om)$ for every $\eta\in H_3$.
\end{proof}

\begin{corollary} \label{DDmap}
Let $\gamma_0:C(\overline{\Om})\to C[-1,0]$ be the partial Dirichlet
trace operator defined by
$$ (\gamma_0 g)(y) \m=\m g(0,y) \FORALL g\in C(\overline{\Om}),\
   y\in [-1,0],$$
and let $D$ be the map defined in Proposition {\rm\ref{def-D}}. Then
$\tilde C_0$ defined by
$$ \tilde C_0\eta \m=\m \gamma_0 D\eta \FORALL \eta\in H_3$$
can be  uniquely extended to a bounded operator $C_0\in\Lscr(H,
L^2[-1,0])$.
\end{corollary}

\begin{proof}
According to Lemma \ref{Deta}, we have
$$ (\tilde C_0\eta)(y) \m=\m \sum_{k\in\nline} \sqrt{\frac{2}{\pi}}
   \frac{\langle\eta,\varphi_k\rangle}{\cosh{k}}\cosh{[k(y+1)]}
   \ \FORALL \eta\in H_3\m,\ y\in[-1,0] \m,$$
which implies that there exists a constant $K>0$ such that
$$ \lVert \tilde C_0\eta\rVert_{L^2[-1,0]}^2 \m\leq\m K \sum_{k\in
   \nline} |\langle\eta,\varphi_k\rangle|^2 =\m K\lVert \eta
   \rVert_{H}^2 \FORALL \eta\in H_3, \vspace{-1.5mm}$$
which shows that $\tilde C_0$ can be extended as claimed.
\end{proof}

\begin{lemma} \label{pei_su_lemma}
The partial Dirichlet map $D$ defined in Proposition {\rm\ref{def-D}}
is bounded from $H_\half$ to $\Hscr^1(\Om)$, i.e. $D\in\Lscr(H_\half,
\Hscr^1(\Om))$. Moreover, \vspace{-1.5mm}
\begin{equation}\label{urma_buna}
   (D\eta)(x,0) \m=\m \eta(x) \FORALL \eta\in H_\half \m,\ \mbox{
   equality in }\ L^2[0,\pi] \m,
\end{equation} \vspace{-1.5mm}
\begin{equation} \label{parti_bune}
   \int_\Om \nabla(D\eta)\cdot \overline{\nabla \Psi} \m\dd x \m\dd
   y \m=\m 0 \FORALL \eta\in H_\half,\ \Psi\in \Hscr^1_{top}(\Om).
\end{equation}
\end{lemma}

\begin{proof} According to Lemma \ref{Deta}, $D\eta$ is given by
\rfb{Detaformula}. Since $\left\{\sqrt{\frac{2}{\pi}}\sin{(kx)}
\right\}_{k\in\nline}$ is an orthnormal basis in $L^2[0,\pi]$, we
have that for every $\eta\in H_\half$, 
$$ \begin{aligned}
    \left\lVert \frac{\partial (D\eta)}{\partial x}\right\rVert
    _{L^2(\Om)}^2 \m&=\m \int_{-1}^0\int_0^\pi\left| \sum_{k\in
	\nline}\sqrt{\frac{2}{\pi}} \frac{k\langle\eta,\varphi_k
	\rangle}{\cosh{k}}\cosh{[k(y+1])}\sin{(kx)}\right|^2\, \dd x\,
    \dd y \\ &\leq\m \sum_{k\in\nline} \frac{\left|k\langle\eta,
	\varphi_k\rangle\right|^2}{\cosh^2{k}} \int_{-1}^0 \cosh^2
    {[k(y+1)]}\, \dd y\\ &=\m \sum_{k\in\nline} \frac{k^2\left|
	\langle\eta,\varphi_k\rangle\right|^2}{\cosh^2{k}}+
    \sum_{k\in\nline} \frac{k\left|\langle\eta,\varphi_k\rangle
	\right|^2}{2\cosh^2{k}}\sinh(2k), \end{aligned}$$
which clearly implies that there exists $K_1>0$ such that
$$ \left\lVert \frac{\partial (D\eta)}{\partial x}\right\rVert
   _{L^2(\Om)} \m\leq\m K_1 \|\eta\|_\half \FORALL \eta\in
   H_\half.$$
A similar estimate for $\|\frac{\partial (D\eta)}{\partial y}\|
_{L^2}$ can be obtained in a completely similar manner. Moreover, we
know from Proposition \ref{def-D} that $\|D\eta\|_{L^2}$ is also
bounded by a similar estimate. Recalling \rfb{Jeremy_is_out} we
conclude that $D\in\Lscr(H_\half,\Hscr^1(\Om))$.

Formula \rfb{urma_buna} in the lemma follows from the last part
of Lemma \ref{Deta} together with Remark \ref{Ephraim_Mirvis} and
the density of $H_3$ in $H_\half$.

To prove \rfb{parti_bune} first we assume that $\eta\in H_3$ so that,
according to Remark \ref{Ephraim_Mirvis}, $D\eta$ is the unique
classical solution of \rfb{LaplaceDeta}. Multiplying the first
equation in \rfb{LaplaceDeta} by $\overline\Psi\in\Hscr^1_{top}(\Om)$
and integrating by parts, it follows that \rfb{parti_bune} holds for
$\eta\in H_3$. Using the density of $H_3$ in $H_\half$ and the fact
that $D\in\Lscr(H_\half,\Hscr^1(\Om))$, it follows that indeed
\rfb{parti_bune} holds for all $\eta\in H_\half$.
\end{proof}

The second important map constructed in this section is a partial
Neumann map. To this aim, recall the space $\Hscr^1_{top}(\Om)$
introduced in \rfb{top_definition} and notice that, due the
version of the Poincar\'e inequality in \cite[Theorem 13.6.9]
{Obs_Book}, the sesquilinear form on $\Hscr^1_{top}(\Om)$ given by
\vspace{-2mm}
\begin{equation} \label{forme_proaste}
   a[f,g] \m=\m \int_\Om \nabla f\cdot\overline{\nabla g}\, \dd x\,
   \dd y \FORALL f,\ g\in\Hscr^1_{top}(\Om)
\end{equation}
defines an inner product on $\Hscr^1_{top}(\Om)$ which is equivalent
to the one inherited from $\Hscr^1(\Om)$. These facts, combined with
the continuity of the Dirichlet trace (as an operator from $\Hscr^1
(\Om)$ to $L^2(\partial\Om)$), imply the following:

\begin{proposition} \label{def-N}
For every $v\in L^2[-1,0]$ there exists a unique function $Nv\in
\Hscr^1_{top}(\Om)$ such that \vspace{-2mm}
\begin{equation} \label{equality_N}
   \int_\Om \nabla (Nv)\cdot \overline{\nabla g}\, \dd x\, \dd y
   \m=\m \int_{-1}^0 v(y) \overline{g(0,y)}\, \dd y \FORALL g\in
   \Hscr^1_{top}(\Om).
\end{equation}
Moreover, the operator $N$, called a {\em partial Neumann map}, is
linear and bounded from $L^2[-1,0]$ to $\Hscr^1_{top}(\Om)$.
\end{proposition}

\begin{proof}
The results follow from the Lax-Milgram theorem by using the
sesqui\-linear form $a[\cdot,\cdot]$ introduced in
\rfb{forme_proaste} (see also \cite[Proposition 7.1]{thinairI}).
\end{proof}

\begin{remark} \label{nicipedeparte} {\rm
The above proposition can be formulated also as follows: for every
$v\in L^2[-1,0]$ the boundary value problem
\begin{equation} \label{Deltaf}
   \left\{ \begin{aligned}
   & \Delta f(x,y) = 0  \quad&\quad((x,y)\in\Om),\\ & f(x,0) = 0,
   \quad \frac{\partial f}{\partial y}(x,-1) = 0\quad&\quad
   (x\in (0,\pi)),\\ & \frac{\partial f}{\partial x}(0,y) = -v,
   \quad \frac{\partial f}{\partial x}(\pi,y) = 0 \quad & \quad
   (y\in (-1,0)), \end{aligned} \right.
\end{equation}
admits a unique weak solution $f=Nv\in\Hscr^1_{top}(\Om)$. If
$f\in C^2(\overline\Om)$ and $v\in C[-1,0]$, then $f=Nv$ is the
unique classical solution of \rfb{Deltaf}. }
\end{remark}

We note that the sequence $(\psi_k)_{k\in\nline}$ defined by
\begin{equation} \label{DEFPSIK}
   \psi_k(y) = \sqrt{2}\cos{\left[(2k-1)\frac{\pi}{2}(y+1)\right]}
   \FORALL k\in\nline,\ y\in [-1,0],
\end{equation}
is an orthonormal basis in $L^2[-1,0]$ (see \cite[[Sect.~2.6]
{Obs_Book}). We can use this basis to construct the scale of 
Hilbert spaces $(\Uscr_\beta)_{\beta\geq 0}$ defined by $\Uscr_0=
L^2[-1,0]$ and (for $\beta>0$)
\begin{equation} \label{Hscrbeta}
   \Uscr_\beta \m=\m \left\{v\in \Uscr_0\ \ \left|\ \ \sum_{k\in
   \nline} (2k-1)^{2\beta} \left|\int_{-1}^0 v(y)\psi_k(y)\, \dd y
   \right|^2 < \infty\right\}\right. \m,
\end{equation}
with the inner products $\left(\langle \cdot,\cdot\rangle_\beta
\right)_{\beta\geq 0}$ given, for every $v,\m\chi\in\Uscr_\beta$,
by \vspace{-1mm}
$$ \langle v,\chi\rangle_{\Uscr_\beta} = \sum_{k\in\nline}
   (2k-1)^{2\beta} \left(\int_{-1}^0 v(y)\psi_k(y)\, \dd y \right)
   \overline{\left(\int_{-1}^0 \chi(y)\psi_k(y)\, \dd y\right)}.$$

\begin{lemma} \label{Nv}
Let $N$ be the operator defined in Proposition {\rm\ref{def-N}}.
Then for every $v\in L^2[-1,0]$ and every $(x,y)\in\Om$ we have
\begin{equation}\label{Nvformula}
   (Nv)(x,y) = \sum_{k\in\nline} a_k\cosh\left[(2k-1)\frac{\pi}
   {2}(x-\pi)\right]\cos{\left[(2k-1)\frac{\pi}{2}(y+1)\right]},
\end{equation}
with convergence in $\Hscr^1_{top}(\Om)$, where \vspace{-1mm}
$$ a_k \m=\m \frac{2\sqrt{2}\langle v,\psi_k\rangle}{(2k-1)\pi
   \sinh{\left[(2k-1)\frac{\pi^2}{2}\right]}} \FORALL k\in\nline.$$
Moreover, for every $v\in\Uscr_2$ we have $Nv\in
C^2(\overline{\Om})$.
\end{lemma}	

\begin{proof}
By using Remark \ref{nicipedeparte} and separation of variables, we
see that
\begin{equation} \label{Gail_skiing}
   (N\psi_k)(x,y) \m=\m \frac{2\psi_k(y)\cosh\left[(2k-1)\frac{\pi}
   {2}(x-\pi)\right]}{(2k-1)\pi\sinh{\left[(2k-1)\frac{\pi^2}{2}
   \right]}} \FORALL k\in\nline,
\end{equation}
for all $(x,y)\in\Om$. Since $Nv=\sum_{k\in\nline}\langle v,\psi_k
\rangle N\psi_k$, this clearly implies \rfb{Nvformula}, with
convergence in $\Hscr^1_{top}(\Om)$ due to Proposition \ref{def-N}.
For every $j\in\{0,1,2\}$,
$$ \left|\frac{\partial^{j,2-j}}{\partial x^j\partial y^{2-j}}
   \left(N\psi_k\right)(x,y)\right|\leq \frac{\sqrt{2}}{2}\pi(2k-1)
   \FORALL k\in\nline,\ (x,y)\in\Om,$$
so that for every $v\in\Uscr$, the series $Nv=\sum_{k\in\nline}
\langle v,\psi_k\rangle N\psi_k$ converges in $C^2(\overline\Om)$ if
the sequence $k\langle v,\psi_k\rangle$ is in $l^1$. For this (by an
argument similar to the one in the proof of Lemma \ref{Deta}) it is
sufficient if the sequence $k^2\langle v,\psi_k\rangle$ is in $l^2$,
which is precisely the condition $v\in\Uscr_2$.
\end{proof}

\section{Partial Dirichlet to Neumann and Neumann to Neumann
         maps} \label{sec5} 

In this section we give an explicit construction of the operators
allowing us to recast \rfb{gravity-1} as a well-posed linear control
system. Recall the orthonormal basis $(\varphi_k)_{k\in\nline}$ in
$H$ introduced in \rfb{mare_definitie} and the corresponding spaces
$H_\alpha$. First we note a direct consequence of Proposition
\ref{def-D} and of Lemma \ref{Deta}.

\begin{corollary} \label{pe_deoarte}
Let $\gamma_1:C^1(\overline\Om)\to C[0,\pi]$ be the partial
Neumann trace operator defined by
$$ (\gamma_1f)(x) \m=\m \frac{\partial f}{\partial y}(x,0) \FORALL
   f\in C^1(\overline\Om),\ x\in [0,\pi].$$
Then $ \tilde A_0$ defined by
$$\tilde A_0 \eta \m=\m \gamma_1 D\eta \FORALL \eta\in H_3,$$
where $D$ is the Dirichlet map defined in Proposition {\rm
\ref{def-D}}, is a linear bounded map from $H_3$ to $C[0,\pi]$.
Moreover, we have
$$ \tilde A_0 \varphi_k \m=\m k\tanh(k)\varphi_k \FORALL k\in\nline.$$
\end{corollary}

We are now in a position to define a partial Dirichlet to Neumann map.

\begin{proposition} \label{prop_def_DN}
The operator $\tilde A_0$ introduced in Corollary {\rm
\ref{pe_deoarte}} has a unique continuous extension to an operator
$A_0:H_1\to H$. This extension is strictly positive and
$\Dscr(A_0^\half)=H_\half$. For each $k\in\nline$, we have
$A_0\varphi_k=\l_k\varphi_k$, where
\begin{equation} \label{valpr0}
   \l_k \m=\m k\tanh(k) \FORALL k\in\nline
\end{equation}
and
\begin{equation} \label{forma1serie}
   A_0\eta \m=\m \sum_{k\in\nline} \l_k \langle \eta,\varphi_k
   \rangle\varphi_k \FORALL \eta\in H_1.
\end{equation}
\end{proposition}

\begin{proof} It is clear from the previous proposition that
$A_0$ fits into the class of diagonalizable operators discussed
in \cite{Obs_Book} in Proposition 3.2.9 and the remarks after it,
and in Proposition 3.4.8 of the same book.
\end{proof}

\begin{proposition} \label{imbarcare}
Let $A_0$ and $D$ be the operators introduced in Propositions
{\rm\ref{prop_def_DN}} and {\rm\ref{def-D}}, respectively. Let
$\gamma\in H_\half$ and $\Psi\in\Hscr^1(\Om)$ be such that
$$\Psi(x,0) \m=\m \gamma(x), \ \mbox{ equality in }\ L^2[0,\pi].$$
Then for every $\eta\in H_\half$ we have $D\eta\in \Hscr^1(\Om)$ and
\begin{equation} \label{mare_de_mare}
   \langle A_0^\half \eta,A_0^\half \gamma\rangle \m=\m \langle
   \nabla (D\eta),\nabla\Psi\rangle_{L^2(\Om)} .
\end{equation}
\end{proposition}

\begin{proof}
First we assume that $\eta\in H_3$, so that according to Lemma
\ref{Deta} we have $D\eta\in C^2(\overline\Om)$. Then
\rfb{mare_de_mare} follows by a simple integration by parts and
Proposition \ref{prop_def_DN}. The fact that $D\eta\in \Hscr^1(\Om)$
for every $\eta\in H_\half$ has already been proved in Lemma
\ref{pei_su_lemma}. Finally, to prove that \rfb{mare_de_mare} still
holds for $\eta\in H_\half$ it suffices to use the density of $H_3$
in $H_\half$, combined with Lemma \ref{pei_su_lemma}. \end{proof}

\begin{corollary} \label{NN-pre}
With $\gamma_1$ as in Corollary {\rm\ref{pe_deoarte}}, define
the operator $\tilde B_1$ by
$$\tilde B_1 v \m=\m \gamma_1 N v \ \FORALL v\in\Uscr_2,$$
where $N$ is the Neumann map introduced in Proposition {\rm
\ref{def-N}}. Then $\tilde B_1$ is a bounded linear operator from
$\Uscr_2$ to $C[0,\pi]$. Moreover, we have
\begin{equation} \label{Hanukka}
   \left(\tilde B_1 \psi_k\right)(x) \m=\m \frac{(-1)^k
   \sqrt{2}}{\sinh\left[(2k-1)\frac{\pi^2}{2}\right]}\cosh\left[
   (2k-1)\frac{\pi}{2} (x-\pi)\right],
\end{equation}
for all $k\in\nline$, $x\in[0,\pi]$, where the functions $\psi_k$
have been defined in \rfb{DEFPSIK}.
\end{corollary}

\begin{proof}
This follows from Lemma \ref{Nv} and the formula \rfb{Gail_skiing}
for $N\psi_k$.
\end{proof}

We are now ready to define a Neumann to Neumann map.

\begin{theorem} \label{NN-map}
The operator $\tilde{B}_1$ introduced in Corollary {\rm\ref{NN-pre}}
can be extended in a unique manner to a linear bounded operator $B_1:
L^2[-1,0]\to L^2[0,\pi]$. Moreover, for every $v\in L^2[-1,0]$ with
$\int_{-1}^0 v(y)\,\dd y=0$ we have that $B_1 v\in H$, where $H$ is
defined in \rfb{H}. Finally,
\begin{multline} \label{idioata}
   \int_0^\pi (B_1 v)(x)\, \overline{\Psi(x,0)} \m\dd x \m=\m
   \int_\Om \nabla(Nv)(x,y)\cdot \overline{\nabla\Psi(x,y)}\, \m\dd
   x \m\dd y\\ -\int_{-1}^0 v(y) \overline{\Psi(0,y)}\, \dd y
   \FORALL \Psi\in\Hscr^1(\Om).
\end{multline}
\end{theorem}

\begin{proof}
For any $v\in L^2[-1,0]$ we set
$$ b_k \m=\m \frac{(-1)^k\sqrt{2}}{\sinh\left[(2k-1)\frac
   {\pi^2}{2}\right]}, \qquad v_k \m=\m \langle v,\psi_k \rangle,$$
and notice that these sequences are in $l^2$ and $\|(v_k)\|_{l^2}=
\|v\|_{L^2[-1,0]}$. From \rfb{Hanukka} it follows that if $v\in
\Uscr_2$ then for every $x\in[0,\pi]$ we have
\begin{equation} \label{B_0vformula}
   (\tilde{B}_1v)(x) \m=\m \sum_{k\in\nline} b_k v_k \cosh\left[
   (2k-1) \frac{\pi}{2}(x-\pi)\right] \m=\m \frac{f(x)+g(x)}{2},
\end{equation}
where
$$ f(x) \m=\m \sum_{k\in\nline} b_k v_k \exp\left[(2k-1)\frac{\pi}
   {2}(x-\pi)\right],$$
\begin{equation} \label{DEFGX}
   g(x) \m=\m \sum_{k\in\nline} b_k v_k \exp\left[(2k-1)\frac{\pi}
   {2}(\pi-x)\right].
\end{equation}
On one hand, from \m $0\leq\exp\left[(2k-1)\frac{\pi}{2}(x-\pi)
\right]\leq 1$ for all $x\in[0,\pi]$, by using Cauchy-Schwarz we 
obtain that there exists $C_1>0$ such that
\begin{equation} \label{jumate_buna}
   \int_0^\pi |f(x)|^2\, \dd x \m\leq\m C_1 \|v\|_{L^2[-1,0]}^2
   \FORALL v\in \Uscr_2 \m.
\end{equation}
On the other hand, from \rfb{DEFGX} it follows that
$$ \int_0^\infty |g(x)|^2\, \m\dd x \m=\m \frac{1}{\pi} \sum_{k,l\in
   \nline} \frac{c_k\overline{c_l} v_k \overline{v_l}}{k+l-1},$$
where \m $c_k=b_k\exp\left[(2k-1)\frac{\pi^2}{2}\right]$ for all
$k\in\nline$.
Using that $|c_k|\leq|c_1|< \sqrt{10}$ for all $k\in\nline$,
together with Hilbert's inequality, see for instance \cite[Chapter
IX]{MR0046395} or the nice survey \cite{Jameson}, we obtain that
\begin{equation} \label{jumate_rea_buna}
   \int_0^\infty |g(x)|^2\, \dd x \m\leq\m 10\,\sum_{k\in\nline}
   |v_k|^2 \FORALL v\in\Uscr_2.
\end{equation}
Putting together \rfb{B_0vformula}, \rfb{jumate_buna} and
\rfb{jumate_rea_buna}, it follows that there exists $C>0$ such that
$$ \lVert \tilde{B}_1v\rVert_{L^2[0,\pi]}^2 \m\leq\m C\lVert v
   \rVert_{L^2[-1,0]}^2 \FORALL v\in\Uscr_2.$$
The above estimate, combined with the density of $\Uscr_2$ in
$L^2[-1,0]$, implies that indeed $\tilde B_1$ admits an unique
extension $B_1\in\Lscr\left(L^2[-1,0],L^2[0,\pi]\right)$.

Assume again that $v\in\Uscr_2$. Then, according to Remark
\ref{nicipedeparte} and to Lemma \ref{Nv} we have that $f=N v$ is a
classical solution of \rfb{Deltaf}, so that for every $v\in\Uscr_2$
we have
$$ \begin{aligned} 0 \m&=\m \int_\Om \Delta(Nv)(x,y) \overline{\Psi
   (x,y)}\m\dd x\m\dd y \m\\ &=\m \int_{-1}^0 v(y)\overline{\Psi
   (0,y)}\m\dd y + \int_0^\pi (B_1v)(x)\overline{\Psi(x,0)}\m\dd x
   -\int_\Om \nabla(Nv)\cdot \overline{\nabla\Psi} \m\dd x \m\dd y.
   \end{aligned}$$
Thus \rfb{idioata} holds for $v\in\Uscr_2$ and by density for $v\in
L^2[-1,0]$. Using that $\int_{-1}^0v(y)\dd y=0$ and taking $\Psi=1$
in \rfb{idioata} we obtain that $B_1v$ indeed satisfies the
condition $\int_0^\pi(B_1 v)(x)\dd x=0$, which implies that
$B_1v\in H$. \end{proof}

\begin{remark} {\rm
An alternative proof of \rfb{jumate_rea_buna} can be given using the
Carleson measure criterion for admissibility, see for instance
\cite[Sect.~5.3]{Obs_Book}. To this aim, consider the Hilbert space
$\tilde X=l^2$, the strictly negative operator $\tilde A={\rm diag}\,
\left(-\frac{(2k-1)\pi}{2}\right)$ and the observation functional
$\tilde C=\bbm{c_1 & c_2 & c_3\ldots}$. Then according to the
aforementioned criterion, $\tilde C$ is an admissible observation operator
for the operator semigroup generated by $\tilde A$, and
\rfb{jumate_rea_buna} follows.}
\end{remark}

The above theorem clearly implies the following result:

\begin{corollary} \label{tildeB_0}
Let $h\in L^2[-1,0]$, with $\int_{-1}^0 h(y)\,\dd y=0$ and let $B_0$
be the operator defined by
$$ B_0 u \m=\m u B_1h \FORALL u\in\cline \m.$$
Then $B_0\in\Lscr(\cline,H)$. Moreover, we have
\begin{equation} \label{adj_con_bis}
   \int_0^\pi (B_0 u)(x) \overline{\Psi(x,0)} \m\dd x =\m u
   \int_\Om \nabla(Nh)\cdot \overline{\nabla\Psi} \m\dd x \m \dd y
   -u \int_{-1}^0 h(y) \overline{\Psi(0,y)}\, \dd y \m,
\end{equation}
for all $u\in\cline$ and $\Psi\in\Hscr^1(\Om)$. In particular,
\begin{equation} \label{adj_con}
   B_0^* \eta \m=\m -\int_{-1}^0 h(y) \overline{(C_0\eta)(y)}\,
   \dd y \FORALL \eta\in H,
\end{equation}
where $C_0=\gamma_0 D$ is the operator introduced in Corollary
{\rm\ref{DDmap}}.
\end{corollary}

\begin{proof}
The fact that $B_0\in\Lscr(\cline,H)$ and \rfb{adj_con_bis} follows
from Theorem \ref{NN-map} (in particular from \rfb{idioata}) with
$v=uh$. Moreover, taking $\Psi=D\eta$ with $\eta\in H_\half$ (see
Lemma \ref{pei_su_lemma}) in \rfb{adj_con_bis}, we see that for every
$u\in\cline$,
$$ \langle B_0 u,\eta\rangle = u\langle B_1 h,\eta\rangle =
   -u \int_{-1}^0 h(y) \overline{(C_0\eta)(y)}\, \dd y + u\nm\int_\Om
   \nabla(Nh) \cdot \overline{\nabla(D\eta)} \dd x\m \dd y .$$
Using \eqref{parti_bune} it follows that the last term in the
right-hand side of the above equation is zero, so that we obtain
\rfb{adj_con}. \end{proof}

\section{Proof of the main results} \label{sem_set} 

Throughout this section we denote by $X$ the Hilbert space
$H_\half\times H$, where $H$ and $(H_\alpha)_{\alpha>0}$ have been
defined in \rfb{H} and \rfb{HALPHA}, respectively. We also introduce
the linear operator $A:\Dscr(A)\to X$ with \ $\Dscr(A)=H_1\times
H_\half$ \ and
\begin{equation} \label{DEFAB2}
   A \m=\m \bbm{0 & I\\ -A_0 & 0}, \ \ \mbox{ i.e., } \ \ A \bbm{
   \varphi\\ \psi} = \bbm{ \psi\\ -A_0\varphi} \FORALL
   \bbm{ \varphi\\ \psi}\in\Dscr(A),
\end{equation}
where $A_0=\gamma_1 D$ is the strictly positive operator on $H$, with
domain $H_1$, which has been introduced in Proposition
\ref{prop_def_DN}. We redefine the inner product on $H_\half$ as 
\vspace{-2mm}
$$ \langle x,z \rangle_\half \m=\m \left\langle A_0^\half x,
   A_0^\half z \right\rangle \m,$$
which is equivalent to the original inner product on $H_\half$. Then
$A$ is skew-adjoint on $X$ (see, for instance, \cite[Proposition
3.7.6]{Obs_Book}), so that, according to Stone's theorem (see, for 
instance \cite[Section 3.7]{Obs_Book}), $A$ generates a group 
$\tline=(\tline_t)_{t\in\rline}$ of unitary operators on $X$. 
Moreover, we recall that $h\in L^2[-1,0]$, with $\int_{-1}^0 h(y)\, 
\dd y=0$ and that we have introduced the input space $U=\cline$. Let 
$B\in\Lscr(U,X)$ be given by \vspace{-3mm}
\begin{equation} \label{defbsingur}
   B \m=\m \begin{bmatrix} 0\\ B_0\end{bmatrix},
\end{equation}
where $B_0\in\Lscr(U,H)$ is as in Corollary \ref{tildeB_0}. Clearly
$B\in\Lscr(U,X)$.

We are now in a position to prove our main well-posedness result:

\begin{proof}[Proof of Theorem {\rm\ref{give_context}}]
With the above notation for $X,\ A,\ U$ and $B$ we consider, for each
$\tau\geqslant 0$, the map $\Phi_\tau$ defined by \rfb{Erdogan}, which
is clearly linear and bounded from $\LE U$ into $X$. Let $z_0=
\sbm{\zeta_0\\ w_0}\in X$, let $u\in L^2_{loc}([0,\infty);U)$ and define $z(t)=\sbm{
\zeta\\ w}\in C([0,\infty);X))$ by \rfb{pancakes}. Then, according to
a classical result (see, for instance, \cite[Remark 4.1.2]{Obs_Book}),
for every $t\geq 0$ and $\psi\in \Dscr(A)$ we have
$$ \langle z(t)-z_0,\psi\rangle_X \m=\m \int_0^t \left[-\langle
   z(\sigma),A\psi\rangle_X + \langle B u(\sigma),\psi\rangle_X
   \right] \dd\sigma \m.$$
Setting $\psi=\sbm{\psi_1\\ \psi_2}$, with $\psi_1\in H_1$ and $\psi_2
\in H_\half$ and using the specific structure \rfb{DEFAB2},
\rfb{defbsingur} of $A$ and $B$, the last formula implies that
\begin{multline} \label{variationala_mare}
   \langle A_0^\half (\zeta(t)-\zeta_0),A_0^\half \psi_1\rangle+
   \langle w(t)-w_0,\psi_2\rangle \m=\m -\left\langle \int_0^t
   A_0^\half \zeta(\sigma) \m\dd\sigma, A_0^\half \psi_2
   \right\rangle\\ + \left\langle \int_0^t w(\sigma)\m \dd\sigma,
   A_0\psi_1\right\rangle + \int_0^t \langle B_0 u(\sigma),\psi_2
   \rangle \m \dd\sigma,
\end{multline}
for every $t\geq 0,\ \psi_1\in H_1,\ \psi_2\in H_\half$. The above
formula holds, in particular, for $\psi_2=0$ and arbitrary $\psi_1
\in H_1$, which yields that
$$ \zeta(t)-\zeta_0 \m=\m \int_0^t w(\sigma)\m \dd\sigma \FORALL
   t\geq 0,$$
so that \m $w(t)=\dot\zeta(t)$, for all $t\geq 0$. Inserting the last
two formulas in \rfb{variationala_mare}, we obtain that
\vspace{-2mm}
\begin{multline} \label{variationala_mare_bis}
   \langle \dot\zeta(t)-w_0,\psi_2\rangle \m=\m -\int_0^t
   \langle  A_0^\half \zeta(\sigma), A_0^\half \psi_2 \rangle
   \m\dd\sigma + \int_0^t\langle B_0 u(\sigma), \psi_2\rangle \m
   \dd\sigma,
\end{multline}
where $t\geq 0,\ w_0=\dot{\zeta}(0)$ and $\psi_2\in H_\half$. (This
formula \rfb{variationala_mare_bis} is the weak form of the equation
$\ddot\zeta=-A_0\zeta+B_0 u$.) Let $\Psi\in\Hscr^1(\Om)$ be such that
$\int_0^\pi\Psi(x,0)\, \dd x=0$ and then $\psi_2(x)=\Psi(x,0)$ is a
function in $H_\half$.  By combining \rfb{variationala_mare_bis} and
\rfb{adj_con_bis} it follows that \vspace{-2mm}
\begin{multline} \label{variationala_mare_bis_bis}
   \langle \dot\zeta(t)-w_0,\psi_2\rangle \m=\m -\int_0^t
   \langle  A_0^\half \zeta(\sigma), A_0^\half \psi_2 \rangle
   \m\dd\sigma \\ + \int_0^t u(\sigma)\, \int_\Om \nabla(Nh)\cdot
   \overline{\nabla\Psi} \m\dd x \m\dd y \m\dd\sigma - \int_0^t
   u(\sigma)\int_{-1}^0 h(y) \overline{\Psi(0,y)} \m\dd y\m\dd\sigma.
\end{multline}
On the other hand, from Proposition \ref{imbarcare} it follows that
$D\zeta(\sigma)\in\Hscr^1(\Om)$ and
$$ \langle  A_0^\half \zeta(\sigma), A_0^\half \psi_2 \rangle \m=\m
   \langle \nabla (D\zeta(\sigma)),\nabla\Psi\rangle \FORALL
   \Psi\in \Hscr^1(\Om),\ \psi_2(x)=\Psi(x,0) \m.$$
The above formula, when combined with \rfb{variationala_mare_bis_bis},
and setting
\begin{equation} \label{Dan_Raviv}
   \phi(t,\cdot,\cdot) \m=\m -[D\zeta(t)](\cdot,\cdot) + u(t)(Nh) 
   (\cdot,\cdot) \FORALL t\geq 0,
\end{equation}
implies that $\phi(t,\cdot, \cdot)\in\Hscr^1(\Om)$ and $(\phi,\zeta)$
satisfies \rfb{weaksol} for every $\Psi\in\Hscr^1(\Om)$ with
$\int_0^\pi\Psi(x,0)\m\dd x=0$. On the other hand, $(\phi,\zeta)$
obviously satisfies \rfb{weaksol} if $\Psi$ is a constant function,
thus $(\phi,\zeta)$ satisfies \rfb{weaksol} for every $\Psi\in\Hscr^1
(\Om)$. Moreover, according to Lemma \ref{pei_su_lemma}, Proposition
\ref{def-N} and the above definition of $\phi$, we have that $\phi\in
L^2_{loc}([0,\infty),\Hscr^1(\Om))$ and \rfb{Jeremy_C} holds, so that
$(\phi,\zeta)$ is a solution of \rfb{gravity-1} in the sense of
Definition \ref{def_solutii}.

Conversely, assume that $(\phi,\zeta)$ is a solution of
\rfb{gravity-1} in the sense of Definition \ref{def_solutii}, with
$\zeta(0)=\zeta_0\in H_\half$ and $\dot\zeta(0)=w_0\in H$. Using
the fact that \rfb{weaksol} holds, in particular, for $\Psi\in
\Hscr^1_{top}(\Om)$ it follows that for every $t\geq 0$ and every
$\Psi\in\Hscr^1_{top}(\Om)$ we have \vspace{-2mm}
$$ \int_\Om \nabla\phi(t,x,y) \cdot \overline{\nabla\Psi(x,y)}\,\dd x
   \m\dd y -u(t)\int_{-1}^0 h(y) \overline{\Psi(0,y)} \,\dd y
   \m=\m 0.$$
Using the notation \vspace{-4mm}
\begin{equation} \label{denottilde}
   \tilde \phi(t,\cdot,\cdot) \m=\m \phi(t,\cdot,\cdot)-u(t) (Nh)
   (\cdot,\cdot),
\end{equation}
where $N$ is the Neumann map defined in Proposition \ref{def-N}, it
follows that
$$ \int_\Om \nabla\tilde\phi(t,x,y) \cdot \overline{\nabla\Psi(x,y)}
   \,\dd x\,\dd y  \m=\m 0 \FORALL \Psi\in\Hscr^1_{top}(\Om).$$

The last formula holds, in particular, for $\Psi\in\Dscr(A_1)$, where
$\Dscr(A_1)$ has been defined in Proposition \ref{AZERO0}, so that an
integration by parts yields that
\begin{equation} \label{will_imply}
   \int_\Om \tilde\phi(t,x,y) \overline{\Delta\Psi(x,y)} \m\dd x
   \m\dd y \m= \int_0^\pi \tilde\phi(t,x,0) \overline{\frac{\partial
   \Psi}{\partial y}(x,0)} \m\dd x\ \ \forall\m \Psi\in\Dscr(A_1).
\end{equation}
Moreover, according to Definition \ref{def_solutii} and Proposition
\ref{def-N}, we have \rfb{Jeremy_C} and \m $(Nh)(x,0)=0$ for
$x\in[0,\pi]$, so that from \rfb{will_imply} it follows that
$$ \int_\Om \tilde\phi(t,x,y) \overline{\Delta\Psi(x,y)} \m\dd x
   \m\dd y \m=\m -\int_0^\pi \zeta(t,x) \overline{\frac{\partial\Psi}
   {\partial y}(x,0)} \m\dd x \FORALL \Psi\in\Dscr(A_1).$$
Comparing the above formula with the definition \rfb{taking} of the
Dirichlet map, with $g=\overline{\Delta\Psi}=-\overline{A_1\Psi}$, 
and recalling that $A_1$ is onto, it follows that
$$ \tilde\phi(t,\cdot,\cdot) \m=\m -[D\zeta(t)](\cdot,\cdot) \FORALL 
   t\geq 0.$$
The last formula and \rfb{denottilde} yield that again 
\rfb{Dan_Raviv} holds.

Now we take $\psi_2\in H_\half$ and we recall from Lemma 
\ref{pei_su_lemma} that $D\psi_2\in\Hscr^1(\Om)$ and that $(D\psi_2)
(x,0)=\psi_2(x)$ for $x\in[0,\pi]$. We can thus choose $\Psi=D\psi_2$
in \rfb{weaksol} and using Proposition \ref{imbarcare} and Corollary
\ref{tildeB_0}, it follows that $\zeta$ satisfies 
\rfb{variationala_mare_bis}. This easily implies that $z=\sbm{\zeta\\ 
\dot\zeta}$ satisfies \rfb{pancakes}.
\end{proof}

In order to prove Theorem \ref{th_mergebine}, we
need the following preliminary result on the eigenvalues and the
eigenvectors of the operator $A$ introduced at the beginning of this
section.

\begin{lemma} \label{lema_spectrala}
Let $(\l_k)_{k\in\nline}$ and $(\varphi_k)_{k\in\nline}$ be the
sequences defined in \rfb{valpr0} and \rfb{mare_definitie},
respectively. We extend the sequences $\mu_k=(\sqrt{\l_k})_{k\in
\nline}$ and $(\varphi_k)_{k\in\nline}$ to $\zline^*$ by setting
\vspace{-4mm}
$$ \m\ \ \mu_{-k} \m=\m -\mu_k, \qquad \varphi_{-k} \m=\m -\varphi_k
   \FORALL k\in\nline.$$

Then the family $\{\phi_k\}_{k\in\zline^*}$ defined by
\begin{equation} \label{Turkey}
   \phi_k \m=\m \frac{1}{\sqrt{2}}\begin{bmatrix} \frac{1}{{\rm i}\mu_k}
   \varphi_k \\ \varphi_k \end{bmatrix} \FORALL k\in\zline^*
\end{equation}
is an orthonormal basis in $X$ formed of eigenvectors of the operator
$A$ defined in \rfb{DEFAB2}. Moreover, for each $k\in\zline^*$, \m
$A\phi_k={\rm i} \mu_k\phi_k$. Finally, there exists $\e>0$ such that for every $\o\in\rline$ with $|\o|\geq 1$, the interval $\left[\o-\frac{
\e}{|\o|},\o+\frac{\e}{|\o|}\right]$ contains at most one element of
the sequence $(\mu_k)_{k\in\zline^*}$.
\end{lemma}

\begin{proof}
According to Proposition \ref{prop_def_DN}, the family
$(\varphi_k)_{k\in\nline}$ defined in \rfb{mare_definitie} is an
orthonormal basis in $H$ formed of eigenvectors of $A_0$ and for each
$k\in\nline$, $A_0\varphi_k=\l_k\varphi_k$, where $(\l_k)_{k\in
\nline}$ have been defined in \rfb{valpr0}. Using the structure
\rfb{DEFAB2} of $A$ and a classical result (see, for instance,
\cite[Section 3.7]{Obs_Book}), it follows that $A$ is diagonalizable,
with the eigenvalues $({\rm i}\mu_k)_{k\in\zline^*}$ corresponding to the orthonormal basis of eigenvectors $(\phi_k)_{k\in\zline^*}$.

Next we prove the last assertion in the lemma by a contradiction
argument. Note that for $k\in\nline$, $\mu_k\approx\sqrt{k}$, with 
exponentially vanishing approximation error. If we assume that the 
assertion in the lemma is false, we obtain the existence of a 
positive sequence $(\o_n)$ with $\o_n\to\infty$ and of a sequence 
$(k_n)$ in $\nline$ with $k_n\to\infty$ such that
\begin{equation} \label{mare_cacat}
   \left\{\mu_{k_n},\ \mu_{k_n+1}\right\} \m\subset\m \left[\o_n-
   \frac{1}{n\o_n},\ \o_n+\frac{1}{n\o_n}\right].
\end{equation}
This implies that \m $\lim_{n\to\infty}\o_n(\mu_{k_n+1}-\mu_{k_n})=
0$. Combining this fact with \rfb{mare_cacat}, it follows that
\vspace{-2mm}
\begin{equation} \label{mare_cacat_bis_bis}
   \lim_{n\to\infty} \mu_{k_n} (\mu_{k_n+1}-\mu_{k_n}) \m=\m 0.
\end{equation}
On the other hand, it is not difficult to check, using \rfb{valpr0}, 
that
$$ \lim_{k\to\infty} \mu_k\ (\mu_{k+1}-\mu_k) \m=\m \half \m,$$
which clearly contradicts \rfb{mare_cacat_bis_bis} and thus ends
the proof. \end{proof}

We are now in a position to prove our main stabilizability result.

\begin{proof}[Proof of Theorem \ref{th_mergebine}]
$1.$ The first assertion follows directly from Curtain and Zwart
\cite[Theorem 5.2.6]{CZ_THE_BOOK}, since $A$ has infinitely many 
unstable eigenvalues. Alternatively, we can apply the main result of
Gibson \cite{Gibson} or Guo, Guo and Zhang 
\cite[Theorem 3]{guo2007lack}.

$2.$ To prove the second assertion, notice that, since the adjoint 
of the operator $B$ defined in \rfb{defbsingur} is $B^*=\bbm{0 & 
B_0^*}$, we see from \rfb{adj_con} and \rfb{Turkey} that
$$ B^*\phi_k \m=\m \frac{-1}{\sqrt{2}} \int_{-1}^0 h(y) \overline
   {(C_0\varphi_k)(y)} \dd y \FORALL k\in\zline^* \m,$$  
where $C_0=\gamma_0 D$. Using \rfb{Dformula}, we get from the above
that
\begin{equation} \label{B*phi}
   B^*\phi_k \m=\m -\frac{1}{\sqrt{\pi}} \int_{-1}^0 h(y)\frac
   {\cosh{[k(y+1)]}}{\cosh{k}} \dd y \FORALL k\in\zline^* \m.
\end{equation}

Assume now that $h$ is a strategic profile, i.e., \rfb{B*phi_start}
holds. Then clearly
$$B^*\phi_k \m\neq \m 0 \FORALL k\in\zline^*.$$
According to \cite[Proposition 6.9.1]{Obs_Book} the pair $(A^*,B^*)$
is approximately observable in infinite time (we have used that the 
eigenvalues of $A$ are distinct). Now it follows from the main result
of Benchimol \cite{benchimol1978note} that the semigroup generated by
$A-BB^*$ is strongly stable.

Conversely, let us assume that $h$ is not a strategic profile, i.e.,
that $h$ does not satisfy assumption \rfb{B*phi_start}. Then from
\rfb{B*phi} there exists a $k\in\nline$ such that $B^*\phi_k=0$. 
Since $A^*=-A$, it follows that for every $F\in\Lscr(X,U)$,
$$(A^*+F^*B^*)\phi_k \m=\m -{\rm i}\mu_k \phi_k \m.$$
Let $\tline^{cl}$ denote the semigroup generated by $A+BF$, we have
$$ \left( \tline_t^{cl} \right)^* \phi_k \m=\m {\rm e}^{-{\rm i}
   \mu_k t} \phi_k \FORALL t\geq 0,$$
which implies that $\left|\langle\tline_t^{cl}\phi_k,\phi_k\rangle
\right|=1$ for all $t\geq 0$, so that $\tline^{cl}$ is not
strongly stable. We have thus shown that if $(A,B)$ is strongly
stabilizable, then $h$ satisfies \rfb{B*phi_start}, which ends the
proof of the second assertion.

$3.$ To prove the third assertion, first we notice that by combining
\rfb{gasitabine} and \rfb{B*phi} it follows that there exists
$M_0>0$ such that
\begin{equation} \label{borne_inf}
   \left|B^*\phi_k\right| \m\geq\m \frac{M_0}{|k|} \FORALL
   k\in\zline^*.
\end{equation}
We introduce, for every $s\in\rline$ and $\delta>0$, the
vector space $WP_{s,\delta}(A)$, called {\em wave
package of frequency $s$ and width $\delta$ associated to the
operator $A$}, which is defined by
$$ WP_{s,\delta}(A) \m=\m \begin{cases}
   \{0\} \ \ \ \ \mathrm{if}\ |\mu_k-s|\geq \delta \ \mathrm{for\
   all}\ k\in\zline^* \m,\\ \text{span}\ \left\{\phi_k\big|~
   k\in\zline^* \ \text{and}\ \ |\mu_k-s|<\delta \right\} \ \
   \mathrm{else.} \end{cases}$$
According to Lemma \ref{lema_spectrala} there exists $\e>0$ such
that, setting
\begin{equation} \label{def_delta}
   \delta(s) \m=\m \frac{\e}{|s|+1} \FORALL s\in\rline,
\end{equation}
we either have that $WP_{s,\delta(s)}(A)=\{0\}$ or
$$ WP_{s,\delta(s)}(A) \m=\m \text{span}\ \{\phi_{k(s)}\},$$
where $k(s)$ is the unique element of $\zline^*$ such that
$$s-\delta(s) \m<\m \mu_{k(s)}< s + \delta(s) .$$
Using the fact that $\mu_k=\sqrt{k\tanh(k)}$ and $\mu_{-k}=-\mu_k$
for $k\in\nline$, together with \rfb{borne_inf}, it follows that
there exists $M_1>0$ such that
$$ |B^*\phi| \m\geq\m \frac{M_1}{(|s|+1)^2} \|\phi\|_X
   \FORALL \phi\in WP_{s,\delta(s)}(A),\ s\in\rline.$$
We have thus obtained that the pair $(A,B)$ satisfies the assumptions
of Theorem 1.1 in \cite{chill2019non} with $\delta$ given by
\rfb{def_delta} and
$$ \gamma(s) \m=\m \frac{M_1}{(|s|+1)^2} \FORALL s\in\rline.$$
We can apply Theorem 1.1 in \cite{chill2019non} to conclude that the
semigroup $\tline^{cl}$ generated by $A-BB^*$ satisfies 
\rfb{final_estimate}. \end{proof}

\begin{remark} {\rm
For the proof of the second assertion we could use (instead of 
Benchimol \cite{benchimol1978note}) the stronger result of Batty and Vu \cite{batty1990stability}, where $A$ generates a contraction semigroup and $B$ is still bounded. An even more general result is in the recently published \cite{Ruth_dead}, where $A$ generates a contraction semigroup and $B$ may be very unbounded (not even admissible).}
\end{remark}

\section*{Acknowledgements}

The authors would like to thank Professor David Lannes (from
Universit\'e de Bordeaux) for valuable advice on this paper. The
authors are members of the ETN network ConFlex, funded by the European
Union's Horizon 2020 research and innovation programme under the Marie
Sklodowska-Curie grant agreement no. 765579. The second author is also
supported by the SingFlows project, grant ANR-18-CE40-0027 of the
French National Research Agency.

\bibliographystyle{elsarticle-num}
\bibliography{Su_Tucsnak}

\begin{thebibliography}{10}
\expandafter\ifx\csname url\endcsname\relax
  \def\url#1{\texttt{#1}}\fi
\expandafter\ifx\csname urlprefix\endcsname\relax\def\urlprefix{URL }\fi
\expandafter\ifx\csname href\endcsname\relax
  \def\href#1#2{#2} \def\path#1{#1}\fi

\bibitem{LiMa}
J.-L. Lions, E.~Magenes, Non-homogeneous {B}oundary {V}alue {P}roblems and
  {A}pplications. {V}ol. {I}, Springer-Verlag, New York, 1972.

\bibitem{Obs_Book}
M.~Tucsnak, G.~Weiss, Observation and {C}ontrol for {O}perator {S}emigroups,
  Birkh\"auser Verlag, Basel, 2009.

\bibitem{whitham2011linear}
G.~B. Whitham, Linear and {N}onlinear {W}aves, Vol.~42, John Wiley \& Sons,
  2011.

\bibitem{lannes2013water}
D.~Lannes, The {W}ater {W}aves {P}roblem: {M}athematical {A}nalysis and
  {A}symptotics, Vol. 188, American Math. Soc., Providence, RI, 2013.

\bibitem{reid1985boundary}
R.~M. Reid, D.~L. Russell, Boundary control and stability of linear water
  waves, SIAM J. on Control and Optim. 23~(1) (1985) 111--121.

\bibitem{mottelet2000controllability}
S.~Mottelet, Controllability and stabilization of a canal with wave generators,
  SIAM J. on Control and Optim. 38~(3) (2000) 711--735.

\bibitem{Weiss1}
G.~Weiss, Admissibility of unbounded control operators, SIAM J. on Control and
  Optim. 27~(3) (1989) 527--545.

\bibitem{BDDM}
A.~Bensoussan, G.~Da~Prato, M.~C. Delfour, S.~K. Mitter, Representation and
  {C}ontrol of {I}nfinite {D}imensional {S}ystems, 2nd Edition, Systems \&
  Control: Foundations \& Applications, Birkh\"auser, Boston, MA, 2007.

\bibitem{chandler2015interpolation}
S.~N. Chandler-Wilde, D.~P. Hewett, A.~Moiola, Interpolation of {H}ilbert and
  {S}obolev spaces: quantitative estimates and counterexamples, Mathematika
  61~(2) (2015) 414--443.

\bibitem{Weiss4}
G.~Weiss, The resolvent growth assumption for semigroups on {H}ilbert spaces,
  J. Math. Analysis and Applic. 145 (1990) 154--171.

\bibitem{Zabcz}
J.~Zabczyk, A note on \m {$C_0$}-semigroups, Bulletin de l'Acad\'emie Polonaise
  des Sciences, S\'erie des sciences math., astr. et phys. 23 (1975) 895--898.

\bibitem{ammari2000stabilization}
K.~Ammari, M.~Tucsnak, Stabilization of {B}ernoulli--{E}uler beams by means of
  a pointwise feedback force, SIAM J. on Control and Optim. 39~(4) (2000)
  1160--1181.

\bibitem{Grisvard}
P.~Grisvard, Elliptic {P}roblems in {N}onsmooth {D}omains, Vol.~24 of
  Monographs and Studies in Mathematics, Pitman, Boston, 1985.

\bibitem{thinairI}
G.~Weiss, M.~Tucsnak, How to get a conservative well-posed linear system out of
  thin air. {P}art {I}: Well-posedness and energy balance, ESAIM Control Optim.
  Calc. Var. 9 (2003) 247--274.

\bibitem{MR0046395}
G.~H. Hardy, J.~E. Littlewood, G.~P\'{o}lya, Inequalities, Cambridge University
  Press, Cambrisge, UK, 1952, 2nd ed.

\bibitem{Jameson}
G.~J.~O. Jameson, Hilbert's inequality and related results, available at
  https://www.maths.lancs.ac.uk/~jameson/hilbert.pdf.

\bibitem{CZ_THE_BOOK}
R.~F. Curtain, H.~Zwart, An {I}ntroduction to {I}nfinite-dimensional {L}inear
  {S}ystems {T}heory, Springer Verlag, New York, 1995.

\bibitem{Gibson}
J.~Gibson, A note on stabilization of infinite dimensional linear oscillators
  by compact feedback, SIAM J. on Control and Optim. 18 (1980) 311--316.

\bibitem{guo2007lack}
F.~Guo, K.~Guo, C.~Zhang, Lack of uniformly exponential stabilization for
  isometric {$C_0$-semigroups} under compact perturbation of the generators in
  {B}anach spaces, Proceedings of the American Math. Soc. 135 (2007)
  1881--1887.

\bibitem{benchimol1978note}
C.~D. Benchimol, A note on weak stabilizability of contraction semigroups, SIAM
  J. on Control and Optim. 16~(3) (1978) 373--379.

\bibitem{chill2019non}
R.~Chill, L.~Paunonen, D.~Seifert, R.~Stahn, Y.~Tomilov, Non-uniform stability
  of damped contraction semigroups, arXiv preprint arXiv:1911.04804.

\bibitem{batty1990stability}
C.~J. Batty, Q.~P. Vu, Stability of individual elements under one-parameter
  semigroups, Transactions of the American Mathematical Society 322~(2) (1990)
  805--818.

\bibitem{Ruth_dead}
R.~Curtain, G.~Weiss, Strong stabilization of (almost) impedance passive
  systems by static output feedback, Mathematical Control and Related Fields
  (MCRF) 10 (2019) 643--671.

\end{thebibliography}

\end{document}